\documentclass[12pt]{amsart}

\usepackage{amssymb}
\usepackage{amscd}
\usepackage{statex}

\title[Galois lines for the quotient curve of the Hermitian curve]{Galois lines for the quotient curve of the Hermitian curve by an involution}
\author{Satoru Fukasawa}

\thanks{The author was partially supported by JSPS KAKENHI Grant Number JP19K03438}
\address{Faculty of Science, Yamagata University, Kojirakawa-machi 1-4-12, Yamagata 990-8560, Japan}
\email{s.fukasawa@sci.kj.yamagata-u.ac.jp}

\newtheorem*{theorem}{Theorem}
\newtheorem*{proposition}{Proposition}

\newtheorem*{lemma}{Lemma}

\newtheorem{corollary}{Corollary}
\newtheorem{fact}{Fact}

\theoremstyle{definition}

\begin{document}
\begin{abstract}
The arrangement of all Galois lines for the quotient curve of the Hermitian curve by an involution in the projective $3$-space is described, in terms of the geometry over finite fields. 
All Galois points for three plane models of this curve admitting three or more Galois points are also determined.  
\end{abstract}
\maketitle

\section{Introduction} 
Let $k$ be an algebraically closed field of characteristic $p \ge 0$ and let $C \subset \mathbb{P}^2$ be an irreducible plane curve of degree $d>1$. 
For a point $P \in \mathbb{P}^2$ defined by linearly independent homogeneous polynomials $F_0, F_1$ of degree one, the rational map 
$$\pi_P: C \dashrightarrow \mathbb{P}^1; \ Q \mapsto (F_0(Q):F_1(Q))$$
is called the projection from $P$. 
The point $P$ is called a Galois point for $C$ if the function field extension $k(C)/\pi_P^*k(\mathbb{P}^1)$ induced by $\pi_P$ is Galois. 
This notion was introduced by Hisao Yoshihara in 1996 (\cite{fukasawa1, miura-yoshihara, yoshihara}). 
In this case, the Galois group is denoted by $G_P$. 
Analogously, the notion of a {\it Galois line} was introduced by Yoshihara.  
A line $\ell \subset \mathbb{P}^3$ defined by linear homogeneous polynomials  $F_0$ and $F_1$ is called a Galois line for a space curve $X \subset \mathbb{P}^3$ if the extension $k(X)/\pi_{\ell}^*k(\mathbb{P}^1)$ induced by the projection 
$$\pi_{\ell}: X \dashrightarrow \mathbb{P}^1; \ P \mapsto (F_0(P): F_1(P))$$ 
from $\ell$ is Galois (\cite{duyaguit-yoshihara, yoshihara2}). 
The Galois group is denoted by $G_\ell$. 
There are only three cases where {\it all} Galois lines for a curve in $\mathbb{P}^3$ of genus $g >1$ are determined (under the assumption that the automorphism group is non-trivial); the Giulietti--Korchm\'{a}ros curve \cite{fukasawa-higashine}, the Artin--Schreier--Mumford curve \cite{fukasawa2}, or the generalized Artin--Schreier--Mumford curve \cite{fukasawa3}. 

In this article, we consider the plane curve $C \subset \mathbb{P}^2$ defined by 
$$ y^{\frac{q+1}{2}}=x^q-x, $$
where the characteristic $p$ is odd and $q \ge 5$ is a power of $p$. 
The smooth model of $C$ is denoted by $X$. 
This curve is well studied (see \cite{stichtenoth1}, \cite[Section 12.1]{hkt}), and is an important maximal curve over $\mathbb{F}_{q^2}$ (see \cite{fgt}, \cite[Theorem 10.41]{hkt}). 
The system of homogeneous coordinates on $\mathbb{P}^3$ is denoted by $(X:Y:Z:W)$. 
It is known that the morphism 
$$ \varphi: X \rightarrow \mathbb{P}^3; \ (x: y : 1: x^2)$$
is an embedding, and that the automorphism group ${\rm Aut}(X)$ acts on the set $\mathcal{C}(\mathbb{F}_q)$ of all $\mathbb{F}_q$-rational points of the conic $\mathcal{C} \subset \mathbb{P}^3$ defined by $Y=X^2-Z W=0$ (see Section 2).   
The description of all Galois lines for the curve $\varphi(X) \subset \mathbb{P}^3$ was asked by Higashine \cite{higashine}. 
This article settles this problem, as follows. 

\begin{theorem} 
Let $\ell \subset \mathbb{P}^3$ be a line.  
Then $\ell$ is a Galois line for $\varphi(X)$ if and only if $\ell$ is in one of the following cases: 
\begin{itemize}
\item[(I)] $\ell \ni (0:1:0:0)$ and $\ell \cap \mathcal{C} \ne \emptyset$; 
\item[(II)] $\ell$ is an $\mathbb{F}_q$-line such that $\ell \ni (0:1:0:0)$ and $\ell \cap \mathcal{C}=\emptyset$;  
\item[(III)] $\ell$ is an $\mathbb{F}_q$-line contained in the plane defined by $Y=0$; 
\item[(IV)] $\ell$ is an $\mathbb{F}_{q^2}$-line such that $\ell=T_{R}\mathcal{C}$ for some $R \in \mathcal{C}(\mathbb{F}_{q^2})\setminus \mathcal{C}(\mathbb{F}_q)$, where $T_R\mathcal{C} \subset \{Y=0\}$ is the tangent line of the conic $\mathcal{C}$ at $R$. 
\end{itemize} 
Furthermore, for all cases, the Galois group $G_\ell$ is determined as follows.   
\begin{itemize} 
\item[(a)] In case (I), $G_\ell \cong C_{\frac{q+1}{2}}$, where $C_\frac{q+1}{2}$ is a cyclic group of order $\frac{q+1}{2}$. 
\item[(b)] In case (II), if there exist different points $R, R' \in \mathcal{C}(\mathbb{F}_q)$ such that $T_{R}\mathcal{C} \cap T_{R'}\mathcal{C}=\ell \cap \{Y=0\}$, then $G_\ell \cong C_{q+1}$. 
\item[(c)] In case (II), if there exists $R \in \mathcal{C}(\mathbb{F}_{q^2})\setminus \mathcal{C}(\mathbb{F}_q)$ such that $T_{R}\mathcal{C} \cap T_{R^q}\mathcal{C}=\ell \cap \{Y=0\}$, where $R^q$ is the image of $R$ under the $q$-Frobenius map, then $G_\ell \cong C_{\frac{q+1}{2}} \times C_2$. 
\item[(d)] In case (III), if $\ell=\overline{RR'}$ for some different points $R, R' \in \mathcal{C}(\mathbb{F}_q)$, then $G_\ell \cong D_{q-1}$, where $\overline{RR'}$ is a line passing through $R, R'$ and $D_{q-1}$ is a dihedral group of order $q-1$. 
\item[(e)] In case (III), if $\ell=T_R\mathcal{C}$ for some $R \in \mathcal{C}(\mathbb{F}_q)$, then $G_{\ell} \cong \mathbb{F}_q$. 
\item[(f)] In case (III), if $\ell=\overline{RR^q}$ for some $R \in \mathcal{C}(\mathbb{F}_{q^2})\setminus \mathcal{C}(\mathbb{F}_q)$, then $G_\ell \cong D_{q+1}$. 
\item[(g)] In case (IV), $G_\ell \cong C_{q+1}$.  
\end{itemize}
\end{theorem}

As an application of Theorem, the arrangements of Galois points for three plane models admitting three or more Galois points are determined. 
These are new examples of plane curves possessing three or more Galois points (\cite{yoshihara-fukasawa}).

\begin{corollary} \label{non-collinear outer Galois points} 
Let $P=(0:1:0)$. 
Then the set of all Galois points for the curve $y^{q+1}=x^q-2x^\frac{q+1}{2}+x$ coincides with the set 
$$ \{P\} \cup \{Y=0\}(\mathbb{F}_q). $$ 
Furthermore, the Galois group at any Galois point is determined as follows. 
\begin{itemize}
\item[(a)] $G_P \cong C_{q+1}$.  
\item[(b)] If $Q=(1:0:0)$ or $(0:0:1)$, then $G_Q \cong \mathbb{F}_q$. 
\item[(c)] If $Q=(\xi:0:1)$ for some $\xi \in \mathbb{F}_q$ with $\xi^{\frac{q-1}{2}}=1$, then $G_Q \cong D_{q-1}$.  
\item[(d)] If $Q=(\eta:0:1)$ for some $\eta \in \mathbb{F}_q$ with $\eta^{\frac{q-1}{2}}=-1$, then $G_Q \cong D_{q+1}$.   
\end{itemize}
\end{corollary}

\begin{corollary} \label{collinear inner Galois points} 
Let $P=(0:1:0)$ and let $Q=(1:0:0)$. 
The set of all Galois points for the curve $y^\frac{q+1}{2}=x^q-x$ coincides with the set 
$$ \{P, Q\} \cup \{(\alpha:0:1) \ | \ \alpha \in \mathbb{F}_q\}. $$
Furthermore, the Galois group at any Galois point is determined as follows. 
\begin{itemize}
\item[(a)] $G_P \cong C_\frac{q+1}{2}$.  
\item[(b)] $G_Q \cong \mathbb{F}_q$. 
\item[(c)] If $R=(\alpha:0:1)$ for some $\alpha \in \mathbb{F}_q$, then $G_R \cong D_{q-1}$. 
\end{itemize} 
\end{corollary} 

\begin{corollary} \label{three outer Galois points} 
There exists a plane model $C' \subset \mathbb{P}^2$ of degree $q+1$ admitting exactly three Galois points $P_1, P_2, P_3$ such that they are collinear outer Galois points, $G_{P_1} \cong G_{P_2} \cong C_{q+1}$ and $G_{P_3} \cong D_{q+1}$.  
\end{corollary}

\section{Preliminaries}
The following fact regarding Galois extensions is needed at several places in the proof of the ``only if'' part of Theorem (see \cite[III.7.1, III.7.2]{stichtenoth2}). 

\begin{fact} \label{Galois extension}
Let $\pi: X \rightarrow X'$ be a surjective morphism of smooth projective curves. 
If the field extension $k(X)/\pi^*k(X')$ is Galois, then the following hold. 
\begin{itemize} 
\item[(a)] The Galois group acts on any fiber of $\pi$ transitively. 
\item[(b)] The ramification index $e_Q$ at $Q$ coincides with the order of the stabilizer subgroup of $Q$ for any point $Q \in X$. 
\item[(c)] If $\pi(Q)=\pi(Q')$ for two points $Q, Q' \in X$, then the ramification indices $e_Q, e_{Q'}$ are the same. 
\end{itemize}
\end{fact} 

The system of homogeneous coordinates on $\mathbb{P}^3$ is denoted by $(X:Y:Z:W)$. 
For $\mathbb{P}^2$ (containing the curve $C$), it is denoted by $(X: Y: Z)$. 
For different points $P, Q \in \mathbb{P}^3$, the line passing through $P$ and $Q$ is denoted by $\overline{PQ}$. 

Hereafter, we consider the curve $C \subset \mathbb{P}^2$, which is the projective closure of the affine curve defined by 
$$ y^\frac{q+1}{2}=x^q-x, $$
and its smooth model $X$. 
See \cite{stichtenoth1} or \cite[Section 12.1]{hkt} for the properties of this curve. 
Let $P_{\infty} \in X$ be the pole of the function $x$ or $y$. 
It is known that 
$$ (x)_\infty=\frac{q+1}{2}P_\infty, \ (y)_\infty=q P_{\infty} $$
(see, for example, \cite[Lemma 12.1]{hkt}). 
The following fact can be confirmed easily (see, for example, \cite[Lemma 12.2]{hkt}). 

\begin{fact} \label{linear system} 
It follows that
$$\mathcal{L}((q+1)P_\infty)=\langle 1, x, y, x^2 \rangle.  $$
In particular, the linear system induced by $\varphi: X \rightarrow \mathbb{P}^3$ is complete. 
\end{fact} 

For $\alpha \in \mathbb{F}_q$, the point of $X$ over the smooth point $(\alpha:0:1) \in C$ is denoted by $P_\alpha$.  
Since 
$$ \varphi=(x:y:1:x^2)=\left(\frac{1}{x}: \frac{y}{x^2}:\frac{1}{x^2}:1\right), $$
it follows that $\varphi(P_\infty)=(0:0:0:1) \in \mathcal{C}$, and that $\varphi(P_{\infty})$ is a smooth point of $\varphi(X)$, namely, $\varphi$ is an embedding. 
Note that 
$$ \mathcal{C}(\mathbb{F}_q)=\varphi(X) \cap \{Y=0\}=\varphi(\{P_{\infty}\} \cup \{P_\alpha \ | \ \alpha \in \mathbb{F}_q \}). $$
For Weierstrass points, the following holds (see \cite[Theorem 12.10]{hkt}, \cite{stichtenoth1}). 

\begin{fact} \label{Weierstrass points} 
The set $ \{P_{\infty}\} \cup \{P_\alpha \ | \ \alpha \in \mathbb{F}_q \}$ coincides with the set of all points $P \in X$ such that the Weierstrass semigroup of $P$ is generated by $(q+1)/2$ and $q$. 
In particular, ${\rm Aut}(X)$ acts on $\mathcal{C}(\mathbb{F}_q)$. 
\end{fact} 

According to Facts \ref{linear system} and \ref{Weierstrass points}, the following holds. 

\begin{fact} \label{automorphisms}
There exists an inclusion 
$$ {\rm Aut}(X) \hookrightarrow PGL(3, k), $$
of which the image preserves the curve $\varphi(X) \subset \mathbb{P}^3$. 
\end{fact} 

It can be confirmed that the tangent line $T_{\varphi(P_\infty)}\varphi(X) \subset \mathbb{P}^3$ at $\varphi(P_\infty)$ is defined by $X=Z=0$, and the tangent line $T_{\varphi(P_\alpha)}\varphi(X)$ at $\varphi(P_\alpha)$ is defined by $X-\alpha Z=W-2\alpha X+\alpha^2 Z=0$, namely, all such lines pass through the point $(0:1:0:0)$. 
If $\sigma \in {\rm Aut}(X) \subset PGL(3, k)$ fixes all points of the set $\mathcal{C}(\mathbb{F}_q)$, then $\sigma$ is represented by the matrix 
$$
\left(\begin{array}{cccc}
1 & 0 & 0 & 0 \\
0 & \zeta & 0 & 0 \\
0 & 0 & 1 & 0 \\
0 & 0 & 0 & 1 
\end{array} \right)
$$ 
for some $\zeta \in k$. 
Considering the defining equation $y^\frac{q+1}{2}=x^q-x$, we have $\zeta^\frac{q+1}{2}=1$.   
Therefore, the following holds. 

\begin{fact} \label{automorphisms2}
The restriction map $\sigma \mapsto \sigma |_{\mathcal{C}(\mathbb{F}_q)}$ induces a homomorphism 
$$ {\rm Aut}(X) \rightarrow {\rm Aut}_{\mathbb{F}_q}(\mathcal{C}) \cong PGL(2, \mathbb{F}_q). $$
Furthermore, the kernel coincides with the set of all automorphisms $(x, y) \mapsto (x, \zeta y)$ with $\zeta^{\frac{q+1}{2}}=1$. 
\end{fact}

Finally, we confirm the following lemma regarding the orders of hyperplanes in $\mathbb{P}^3$. 

\begin{lemma} \label{tangent hyperplane} 
For any point $P \in X$ with $\varphi(P) \in \mathcal{C}(\mathbb{F}_q)$, there exists a unique hyperplane $H \subset \mathbb{P}^3$ such that 
$$ {\rm ord}_P\varphi^*H=q+1. $$
Furthermore, the plane $H$ is defined over $\mathbb{F}_q$. 
\end{lemma} 

\begin{proof}
For the point $P_\infty$, the hyperplane is defined by $Z=0$. 
For a point $P_\alpha$, the hyperplane is defined by $W-2\alpha X+\alpha^2 Z=0$. 
\end{proof}

\section{The ``only if'' part of the proof of Theorem} 
In this section, we assume that $\ell \subset \mathbb{P}^3$ is a Galois line for $\varphi(X)$. 

{\it Case: $\ell \not \subset \{Y=0\}$.} 
Let $Q$ be a point given by $\ell \cap \{Y=0\}$. 
For any point $R \in \mathcal{C}(\mathbb{F}_q)$ with $R \ne Q$, the number of points of $\overline{QR} \cap \varphi(X) \subset \mathcal{C}(\mathbb{F}_q)$ is at most two. 
If $R=Q$, then the fiber of $\pi_\ell \circ \varphi$ containing $R$ is unique. 
(We can ignore such a fiber below.)
Let $V_R$ be the plane spanned by $\ell$ and $R$. 
It follows from Fact \ref{Weierstrass points} that $G_\ell$ preserves $\mathcal{C}(\mathbb{F}_q)$, and that $\sigma(R) \in V_R$ for any $\sigma \in G_\ell$. 
By these facts and Fact \ref{Galois extension} (a), there exists a point $R \in \mathcal{C}(\mathbb{F}_q)$ such that the fiber of $\pi_\ell\circ\varphi$ containing $R$ coincides with the set $\overline{QR} \cap \mathcal{C}(\mathbb{F}_q)$ or with $\{R\}$.   
Since $X$ is not rational or hyperelliptic (see, for example, \cite[Lemma 12.2]{hkt}), it follows from Fact \ref{Galois extension} (c) that tangent line $T_R\varphi(X)$ at $R$ is contained in the plane $V_R$. 
We can take two points $R, R' \in \mathcal{C}(\mathbb{F}_q)$ such that $V_R \cap V_{R'}=\ell$. 
Since $(0:1:0:0) \in T_R\varphi(X) \cap T_{R'}\varphi(X) \subset V_R \cap V_{R'}$, it follows that $(0:1:0:0) \in \ell$. 

Assume that $\ell$ is not in case (I), namely, $\ell \ni (0:1:0:0)$ and $\ell \cap \mathcal{C}=\emptyset$. 
We prove that $Q$ is an $\mathbb{F}_q$-rational point. 
Assume by contradiction that $Q$ is not an $\mathbb{F}_q$-rational point. 
There exist at least $q-1$ points $R \in \mathcal{C}(\mathbb{F}_q)$  such that $\overline{QR} \cap \mathcal{C}(\mathbb{F}_q)=\{R\}$. 
Let $V_R$ be the plane spanned by $\ell$ and $R \in \mathcal{C}(\mathbb{F}_q)$. 
Since $\sigma(\mathcal{C}(\mathbb{F}_q))=\mathcal{C}(\mathbb{F}_q)$ and $\sigma(R) \in V_R$ for any $\sigma \in G_\ell$, it follows from Fact \ref{Galois extension} (a) that $G_\ell$ fixes $R$. 
This implies that $G_\ell$ fixes $q-1 \ge 4$ points on $\mathcal{C}$, and that $G_\ell$ fixes all points of the plane defined by $Y=0$. 
It follows from Fact \ref{automorphisms2} that $|G_\ell| \le \frac{q+1}{2}$, and that $\ell \cap \varphi(X) \ne \emptyset$. 
Since $\ell \ni (0:1:0:0)$, it follows that $Q \in \mathcal{C}$. 
This is a contradiction. 
Therefore, $\ell$ is in case (II). 

{\it Case: $\ell \subset \{Y=0\}$.} 
Note that $\sigma |_{\{Y=0\}} \ne 1$ for any element $\sigma \in G_\ell \setminus \{1\}$, by Fact \ref{Galois extension} (b).  
If the set $\ell \cap \mathcal{C}(\mathbb{F}_q)$ contains two points, then $\ell$ is an $\mathbb{F}_q$-line. 
Assume that $\ell \cap \mathcal{C}(\mathbb{F}_q)$ consists of a unique point $R$ (maybe $\ell \cap \mathcal{C}=\{R, R'\}$ for some point $R' \not\in \mathcal{C}(\mathbb{F}_q)$), namely, $\ell \cap \varphi(X)=\{R\}$. 
Since $\ell \ne T_R\varphi(X)$, it follows that $\deg \pi_\ell=q$.   
Let $V_R$ be the plane spanned by $T_R\varphi(X)$ and $\ell$.  Then $\sigma(R) \in V_R$ for any $\sigma \in G_\ell$. 
Since $V_R \cap \mathcal{C}(\mathbb{F}_q)=\{R\}$, it follows that $G_\ell$ fixes $R$. 
It follows from Fact \ref{Galois extension} (b) that ${\rm ord}_RV_R=q+1$. 
From Lemma, $V_R$ is defined over $\mathbb{F}_q$. 
Therefore, the line $\ell=V_R \cap \{Y=0\}$ is an $\mathbb{F}_q$-line and $\ell \cap \mathcal{C}=\{R\}$. 
In this case, $\ell=T_R\mathcal{C}$. 

Assume that $\ell \cap \mathcal{C}=\{R, R'\}$ for some $R, R' \in \mathcal{C} \setminus \mathcal{C}(\mathbb{F}_q)$ with $R \ne R'$, namely, $\ell \cap \varphi(X)=\emptyset$. 
Then $G_\ell$ acts on $\ell$. 
This implies that there exists $\sigma \in G_\ell \setminus \{1\}$ such that $\sigma(R)=R$ and $\sigma(R')=R'$. 
Since $\sigma$ acts on the conic $\mathcal{C}$ defined over $\mathbb{F}_q$, it follows that $R, R'$ are $\mathbb{F}_{q^2}$-rational points with $R^q=R'$, where $R^q$ is the image of $R$ under the $q$-Frobenius map. 
This implies that $\ell=\overline{RR'}=\overline{RR^q}$ is an $\mathbb{F}_q$-line. 
Case (III) is obtained. 

Assume that $\ell \cap \mathcal{C}=\{R\} \not\subset \mathcal{C}(\mathbb{F}_q)$, namely, $\ell$ is the tangent line $T_R\mathcal{C}$ at $R$. 
Then $G_\ell$ acts on $\ell$ and $\mathcal{C}$. 
This implies that $G_\ell$ fixes $R$. 
Since $G_\ell$ acts on the conic $\mathcal{C}$ defined over $\mathbb{F}_q$ and does not fix any point of $\mathcal{C}(\mathbb{F}_q)$, it follows that $R \in \mathcal{C}(\mathbb{F}_{q^2})\setminus \mathcal{C}(\mathbb{F}_q)$.  
Then $\ell$ is in case (IV). 
 
\section{The ``if'' part of the proof of Theorem} 
We consider case (I). 
Let $\ell$ be a line with $\ell \cap \mathcal{C} \ne \emptyset$. 
Then $\ell$ is defined by $X=Z=0$ or $X-\alpha Z=W-\alpha^2 Z$ for some $\alpha \in k$. 
For the case $X=Z=0$, we have an extension $k(x, y)/k(x)$. 
For the case $X-\alpha Z=W-\alpha^2 Z=0$, we have an extension $k(x, y)/k(x+\alpha)$. 
In both cases, $\ell$ is a Galois line with $G_\ell \cong C_\frac{q+1}{2}$. 

We consider case (II). 
Let $Q$ be an $\mathbb{F}_q$-rational point defined by $\ell \cap \{Z=0\}$. 
Assume that there exist different points $R, R' \in \mathcal{C}(\mathbb{F}_q)$ such that $T_R\mathcal{C} \cap T_{R'}\mathcal{C}=\{Q\}$. 
We can assume that $T_R\mathcal{C}$ and $T_{R'}\mathcal{C}$ are defined by $Y=Z=0$ and $Y=W=0$ respectively, namely, $Q=(1:0:0:0)$ and $\ell$ is defined by $Z=W=0$. 
Then 
$\pi_\ell \circ \varphi=(1:x^2)$. 
Let $t=x^2$. 
Then 
$$ y^{q+1}=(x^q-x)^2=x^{2q}-2x^{q+1}+x^2=t^q-2t^\frac{q+1}{2}+t. $$
Since the rational map $(x, y) \mapsto (t, y)$ is separable and generically one to one (according to the equation $y^\frac{q+1}{2}=x^q-x$), it follows that $k(x, y)=k(y, t)$.  
These imply that the extension $k(x, y)/k(t)=k(y, t)/k(t)$ is Galois and $G_\ell \cong C_{q+1}$. 

Let $\alpha \in \mathbb{F}_{q^2} \setminus \mathbb{F}_q$ and let $R=(\alpha:0:1:\alpha^2)$. 
Then $R^q=(\alpha^q:0:1:\alpha^{2q})$, the point $Q$ given by $T_R\mathcal{C} \cap T_{R'}\mathcal{C}$ is given by 
$$ \left(\frac{\alpha+\alpha^q}{2}:0:1:\alpha^{q+1} \right). $$
We consider the line $\ell$ defined by 
$$ X-\frac{\alpha+\alpha^q}{2}Z=W-\alpha^{q+1} Z=0. $$
Then 
$$ \pi_\ell \circ \varphi=\left(1:\frac{x^2-\alpha^{q+1}}{x-\frac{\alpha+\alpha^q}{2}}\right). $$
Let $c \in k$ be such that 
$$ c^{\frac{q+1}{2}}=-(\alpha-\alpha^q)^2. $$
We consider the rational map 
$$ \sigma_c: C \dashrightarrow \mathbb{P}^2; \ (x, y) \mapsto \left( \frac{(\alpha+\alpha^q)x-2\alpha^{q+1}}{2x-(\alpha+\alpha^q)}, \ \frac{c y}{\{2 x-(\alpha+\alpha^q)\}^2}\right). $$
It can be confirmed that $\sigma_c(C) \subset C$, and that $\sigma_c$ fixes the function $\frac{x^2-\alpha^{q+1}}{x-\frac{\alpha+\alpha^q}{2}}$. 
Since the rational map $(x, y) \mapsto (x, \zeta y)$ with $\zeta^{\frac{q+1}{2}}=1$ acts on $C$ and fixes the function above, it follows that $\ell$ is a Galois line. 
We take $c_0=(\alpha-\alpha^q)^2$. 
It can be confirmed that $\sigma_{c_0}^2=1$. 
It is easily verified that $G_\ell$ is abelian. 
Therefore, $G_\ell \cong C_{\frac{q+1}{2}} \times C_2$. 

We consider case (III). 
Let $\ell$ be an $\mathbb{F}_q$-line such that $\ell=\overline{RR'}$ for some $R, R' \in \mathcal{C}(\mathbb{F}_q)$. 
We can assume that $\ell$ is defined by $X=Y=0$. 
For the projection form $\ell$, we have an extension $k(x, y)=k(y/x)$. 
We take $t=y/x$. 
Then we have the relation
$$ f(x, t):=x^{q-1}-t^{\frac{q+1}{2}}x^\frac{q-1}{2}-1=0. $$
Let $\zeta$ be primitive $\frac{q-1}{2}$-th root of unity, and let $\xi^\frac{q-1}{2}=-1$. 
Then $f(\zeta x, t)=0$ and $f(\xi/x, t)=0$. 
Since all roots of $f(x, t) \in k(t)[x]$ are contained in $k(x, t)$, it follows that $\ell$ is a Galois line. 
It is easily confirmed that $G_\ell \cong D_{q-1}$.  

Let $\ell=T_R\mathcal{C}$ for some $R \in \mathcal{C}(\mathbb{F}_q)$. 
We can assume that $R=(0:0:1:0)$ and $\ell=T_R\mathcal{C}$ is defined by $Y=W=0$. 
Then we have an extension $k(x, y)/k(y)$ and the relation $x^q-x-y^\frac{q+1}{2}=0$. 
This extension is obviously Galois and $G_\ell \cong \mathbb{F}_q$. 

Let $\alpha \in \mathbb{F}_{q^2} \setminus \mathbb{F}_q$ and let $R=(\alpha:0:1:\alpha^2)$. 
Then $R^q=(\alpha^q:0:1:\alpha^{2q})$, and $\ell=\overline{RR^q}$ is defined by 
$$ Y=W-(\alpha^q+\alpha)X+\alpha^{q+1}Z=0.  $$
Then
$$ \pi_\ell \circ \varphi=\left( \frac{x^2-(\alpha^q+\alpha)x+\alpha^{q+1}}{y}:1\right)$$
Let $e^{\frac{q+1}{2}}=1$. 
We consider the rational map 
$$ \sigma: C \dashrightarrow \mathbb{P}^2; \ (x, y) \mapsto \left( \frac{(\alpha-e \alpha^q)x-(1-e)\alpha^{q+1}}{(1-e) x+(-\alpha^q+e \alpha)}, \ \frac{e (\alpha^q-\alpha)^2y}{((1-e) x+(-\alpha^q+e \alpha))^2}\right). $$
It can be confirmed that $\sigma(C) \subset C$, and that $\sigma$ fixes the function $\frac{x^2-(\alpha^q+\alpha)x+\alpha^{q+1}}{y}$. 
Let $f^{\frac{q+1}{2}}=-1$. 
We consider the rational map 
$$ \tau: C \dashrightarrow \mathbb{P}^2; \ (x, y) \mapsto \left( \frac{(\alpha^q-f \alpha)x+(-\alpha^{2q}+f \alpha^2)}{(1-f) x+(-\alpha^q+f \alpha)}, \ \frac{f (\alpha^q-\alpha)^2y}{((1-f) x+(-\alpha^q+f \alpha))^2}\right). $$
It can be confirmed that $\tau(C) \subset C$, and that $\tau$ fixes the function $\frac{x^2-(\alpha^q+\alpha)x+\alpha^{q+1}}{y}$. 
These imply that $\ell$ is a Galois line. 
For the automorphism $\sigma$, the image $\sigma((x:0:1:x^2))$ of any point $(x:0:1:x^2) \in \mathcal{C}$ is computed as 
\begin{eqnarray*} 
& & ((\alpha-e \alpha^q) x-(1-e)\alpha^{q+1})((1-e) x+(-\alpha^q+e \alpha)) :  0  : \\
& & \hspace{20mm} ((1-e) x+(-\alpha^q+e \alpha))^2:((\alpha-e \alpha^q) x-(1-e)\alpha^{q+1})^2).  
\end{eqnarray*}
For the automorphism $\tau$, the image $\tau((x:0:1:x^2))$ of any point $(x:0:1:x^2) \in \mathcal{C}$ is computed as 
\begin{eqnarray*} 
& & ((\alpha^q-f \alpha) x+(-\alpha^{2q}+f\alpha^2))((1-f) x+(-\alpha^q+f \alpha)) :  0  : \\
& & \hspace{20mm} ((1-f) x+(-\alpha^q+f \alpha))^2:((\alpha^q-f \alpha) x+(-\alpha^{2q}+f \alpha^2))^2).  
\end{eqnarray*}
Since $\sigma(R)=R$, $\sigma(R^q)=R^q$, $\tau(R)=R^q$ and $\tau(R^q)=R$, it follows that $G_\ell \cong D_{q+1}$. 

We consider case (IV). 
Let $\alpha \in \mathbb{F}_{q^2}\setminus \mathbb{F}_q$, and let $R=(\alpha:0:1:\alpha^2)$. 
Then the tangent line $T_R\mathcal{C}$ is defined by 
$$ Y=-2\alpha X+\alpha^2 Z+W=0. $$
Assume that $\ell=T_R\mathcal{C}$. 
Then 
$$ \pi_\ell \circ \varphi=\left(\frac{(x-\alpha)^2}{y}:1\right). $$
Let $e^{q+1}=1$. 
We consider the rational map 
$$ \sigma: C \dashrightarrow \mathbb{P}^2; \ (x, y) \mapsto \left( \frac{(\alpha-e \alpha^q)x-(1-e)\alpha^{q+1}}{(1-e) x+(-\alpha^q+e \alpha)}, \ \frac{e^2 (\alpha^q-\alpha)^2y}{((1-e) x+(-\alpha^q+e \alpha))^2}\right). $$
It can be confirmed that $\sigma(C) \subset C$, and that $\sigma$ fixes the function $\frac{(x-\alpha)^2}{y}$. 
For the automorphism $\sigma$,
the image $\sigma((x:0:1:x^2))$ of any point $(x:0:1:x^2) \in \mathcal{C}$ is computed as 
\begin{eqnarray*} 
& & ((\alpha-e \alpha^q) x-(1-e)\alpha^{q+1})((1-e) x+(-\alpha^q+e \alpha)) :  0  : \\
& & \hspace{20mm} ((1-e) x+(-\alpha^q+e \alpha))^2:((\alpha-e \alpha^q) x-(1-e)\alpha^{q+1})^2).  
\end{eqnarray*}
Since $\sigma$ fixes $R$ and $R^q$, it follows that $\ell$ is a Galois line with $G_\ell \cong C_{q+1}$. 

\section{Plane models admitting three or more Galois points}
We consider the projection $\pi_R$ from a point $R=(1:0:0:0) \in \mathbb{P}^3$.
Then
$$ \pi_R \circ \varphi (x, y)=(y:1:x^2). $$
As we saw the proof of the ``if'' part of Theorem, case (II), it follows that the rational map $\pi_R \circ \varphi$ is birational onto its image. 
Let $t=x^2$. 
Then 
$$ y^{q+1}=t^q-2t^\frac{q+1}{2}+t, $$
and the curve $(\pi_R \circ \varphi)(X)$ is defined by this equation. 
According to Theorem, there exist exactly $q+2$ Galois lines $\ell$ with $\ell \ni R$. 
These lines correspond to all Galois points for $(\pi_R \circ \varphi)(X)$. 
Since 
$$ t^q-2t^{\frac{q+1}{2}}+t=t(t^{\frac{q-1}{2}}-1)^2, $$
the set of all singular points coincides with the set $\{(\xi:0:1) \ | \ \xi^{\frac{q-1}{2}}=1\}$. 
For any element $\gamma \in \mathbb{F}_q \setminus \{0\}$, $\gamma^{\frac{q+1}{2}}=1$ or $\gamma^{\frac{q+1}{2}}=-1$.  
The proof of Corollary \ref{non-collinear outer Galois points} is completed. 

We consider the projection 
$$ (x:y:1:x^2) \mapsto (x:y:1) $$
from a point $(0:0:0:1) \in \mathbb{P}^3$. 
This is obviously birational, and the image coincides with the plane model $C$.  
According to Theorem, there exist exactly $q+2$ Galois lines $\ell$ with $\ell \ni (0:0:0:1)$. 
These lines correspond to all Galois points for $C$. 
The proof of Corollary \ref{collinear inner Galois points} is completed.  

Let $\alpha, \beta \in \mathbb{F}_{q^2} \setminus \mathbb{F}_q$ with $\alpha \ne \beta$, and let $R=(\alpha: 0: 1: \alpha^2)$, $R'=(\beta: 0 :1: \beta^2)$. 
Then the point $Q$ given by $T_R\mathcal{C} \cap T_{R'}\mathcal{C}$ is represented by 
$$\left(\frac{\alpha+\beta}{2}:0:1:\alpha \beta\right). $$
Then lines $T_R\mathcal{C}$ and $T_{R'}\mathcal{C}$ are Galois lines such that the Galois groups are cyclic groups of order $q+1$. 
Assume that $\beta \ne \alpha^q$. 
Since $Q$ is an $\mathbb{F}_{q^2}$-rational point but not an $\mathbb{F}_q$-rational point, there exists a unique line $\ell \ni Q$ defined over $\mathbb{F}_q$. 
According to Theorem, $\ell$ is a Galois line. 
As we saw in the proof of ``if'' part of Theorem, case (IV), it follows that  $G_{T_R\mathcal{C}}$ fixes exactly two points $(\alpha:0:1:\alpha^2)$ and $(\alpha^q:0:1:\alpha^{2q})$ on the conic $\mathcal{C}$. 
This implies that $G_{T_R\mathcal{C}} \cap G_{T_{R'}\mathcal{C}}=\{1\}$, and that the projection from $Q$ is birational onto its image.  
Therefore, the plane model obtained as the image of the projection from $Q$ possesses exactly three Galois points. 

Let $Q'=(1:0:0:0)$. 
The line $\overline{QQ'}$ is defined by $W-\alpha\beta Z=0$. 
We take $\gamma \in \mathbb{F}_q$ such that $\gamma \ne \delta^2$ for any $\delta \in \mathbb{F}_q$, and $\beta=\gamma/\alpha$. 
Then $\pm \sqrt{\gamma} \in \mathbb{F}_{q^2} \setminus \mathbb{F}_q$, and two points $(\pm \sqrt{\gamma}:0:1:\gamma)$ are contained in an $\mathbb{F}_q$-line $\overline{QQ'}$. 
According to Theorem, the line $\overline{QQ'}$ is a Galois line with $G_{\overline{QQ'}} \cong D_{q+1}$. 
Considering the projection from $Q$, we have Corollary \ref{three outer Galois points}. 

If we take $\alpha^{q+1} \ne 1$ and $\beta=1/\alpha$, then $\mathcal{C} \cap \overline{QQ'}$ consists of two points $(1:0:1:1), (-1:0:1:1) \in \mathcal{C}(\mathbb{F}_q)$.  
We have the following. 

\begin{proposition} 
There exists a plane model $C'' \subset \mathbb{P}^2$ of degree $q+1$ admitting exactly three Galois points $P_1, P_2, P_3$ such that they are collinear Galois points, $G_{P_1} \cong G_{P_2} \cong C_{q+1}$ and $G_{P_3} \cong D_{q-1}$.  
\end{proposition}

\begin{center}
{\bf Acknowledgments} 
\end{center} 
The author is grateful to Doctor Kazuki Higashine for helpful discussions.

\end{document}